\documentclass[final,leqno]{siamltex}

\usepackage{amsfonts}
\usepackage{graphicx}
\usepackage{epstopdf}
\ifpdf
  \DeclareGraphicsExtensions{.eps,.pdf,.png,.jpg}
\else
  \DeclareGraphicsExtensions{.eps}
\fi

\usepackage{amsmath,amssymb}
\usepackage{listings,xcolor}
\usepackage{caption, subcaption}
\usepackage{algorithm,algpseudocode}
\algrenewcommand\algorithmicrequire{\textbf{Input:}}
\algrenewcommand\algorithmicensure{\textbf{Output:}}
\usepackage{multirow, multicol}
\usepackage{bm}
\usepackage{newclude}
\usepackage{cases}
\usepackage{lineno}
\newtheorem{remark}{Remark} 
\newtheorem{assumption}{Assumption} 
\begin{document}

\title{Improved high-index saddle dynamics for finding saddle points and solution landscape\thanks{This work was supported by the National Natural Science Foundation of China (No.12225102, T2321001, 12288101, 12301520, 12301555), the Taishan Scholars Program of Shandong Province (No. tsqn202306083), and the National Key Research and Development Program of China (No. 2023YFA1008903).}}

\author{Hua Su\thanks{Beijing International Center for Mathematical Research, Peking University, Beijing, 100871, China. ({suhua@pku.edu.cn}).}
\and Haoran Wang\thanks{School of Mathematical Sciences, Peking University, Beijing, 100871, China.  ({sdrz\_whr@163.com}).}
\and Lei Zhang\thanks{Corresponding author. Beijing International Center for Mathematical Research, Center for Quantitative Biology, Center for Machine Learning Research, Peking University, Beijing, 100871, China. ({zhangl@math.pku.edu.cn}).}
\and Jin Zhao\thanks{Academy for Multidisciplinary Studies, Capital Normal University, and Beijing National Center for Applied Mathematics, Beijing, 100048, China.
  ({zjin@cnu.edu.cn}).}
\and Xiangcheng Zheng\thanks{School of Mathematics, Shandong University, Jinan, 250100, China. ({xzheng@sdu.edu.cn})}
}

\maketitle

\begin{abstract}
We present an improved high-index saddle dynamics (iHiSD) for finding saddle points and constructing solution landscapes, which is a crossover dynamics from gradient flow to traditional HiSD such that the Morse theory for gradient flow could be involved. We propose analysis for the reflection manifold in iHiSD, and then prove its stable and nonlocal convergence from outside of the region of attraction to the saddle point, which resolves the dependence of the convergence of  HiSD on the initial value. We then present and analyze a discretized iHiSD that inherits these convergence properties. Furthermore,
based on the Morse theory, we prove that any two saddle points could be connected by a sequence of trajectories of iHiSD. Theoretically, this implies that a solution landscape with a finite number of stationary points could be completely constructed by means of iHiSD, which partly answers the completeness issue of the solution landscape for the first time and indicates the necessity of integrating the gradient flow in HiSD. Different methods are compared by numerical experiments to substantiate the effectiveness of the iHiSD method.
\end{abstract}



\pagestyle{myheadings}
\thispagestyle{plain}
\markboth{Improved High-index Saddle Dynamics}{H. Su, H. Wang, L. Zhang, J. Zhao and X. Zheng}

\section{Introduction}
Finding stationary points of complex systems has attracted increasing attentions in various applications such as protein folding \cite{protein1}, superconductivity \cite{superconductivity} and liquid crystals \cite{shi2024}. The solution landscape of complex systems, which consists of all stationary points and their connections \cite{HiSD_PRL}, provides a comprehensive perspective and helps to understand the underlying mechanisms, see e.g. \cite{shi2023nonlinearity, shi2022siap,Wang_Acta2021}. 
%These functions map configuration spaces to energy values, where stable states correspond to local minima and transition states are marked by index-1 saddle points. Herein, the Morse index of a stationary point is defined as the count of negative eigenvalues of the Hessian matrix\cite{Morse_Milnor}. % define later % ref. LY
 %\cite{farrell2015,li2001minimax,quapp2014locating}.
Mathematically, a stationary point $x^*$ of an energy function $E(x)$ satisfies $\nabla E(x^*) = 0$. 
In the context of Morse theory \cite{Morse_Milnor}, the stationary points are further classified by the Morse index, which is defined as the negative inertia index of the Hessian matrix $G(x^*) = \nabla\nabla E(x^*)$. In practice, the complex geometric structures of the problems and the unstable nature of saddle points introduce significant challenges in numerical computations, which motivates extensive investigations on designing searching algorithms.

In general, the existing methods of locating saddle points include,  but are not limited to, the path-finding methods such as the nudged elastic band method \cite{nudged,Jnsson1998NudgedEB} and the string method \cite{string,stringM2007,LiuChe}, the walker-type methods including dimer-type methods \cite{GouOrt,dimer}, gentlest ascent dynamics (GAD) \cite{gad,WittenLaplacians,ShrinkingDimer} and the high-index saddle dynamics (HiSD) \cite{HiSD_PRL,yin2019high,zhang2022sinum,csiam2023}, the iterative minimization formulations \cite{GaoLen,GuZho} and the minimax methods \cite{LiuXieSCM,LiuXie}.
%In particular, HiSD generalizes GAD, which searches for index-$1$ saddle points, to higher-index saddle points. 
In particular, the HiSD for locating an index-$k$ saddle point determines the search direction by reflecting the gradient along the unstable subspace, as described in the following system
\begin{subnumcases}{\label{HiSD0}}
    \frac{\mathrm{d}x}{\mathrm{d}t}=-\beta\left(I-2\sum_{j=1}^{k}v_jv_j^T\right)\nabla E(x), \label{HiSD_intro0_x}\\
    \frac{\mathrm{d}v_i}{\mathrm{d}t}=-\gamma\left(I-v_iv_i^T-2\sum_{j=1}^{i-1}v_jv_j^T\right)G(x)v_i,\quad 1\le i\le k, \label{HiSD_intro0_v}
\end{subnumcases}
where $\beta, \gamma > 0$ are hyperparameters that adjust the time scale.
By introducing the reflection matrix 
$$R=I-2\sum_{j=1}^{k}v_jv_j^T,$$
 the above system \eqref{HiSD0} is equivalent to the following dynamics, which has a more compact form
\begin{subnumcases}{\label{HiSD}}
    \frac{\mathrm{d}x}{\mathrm{d}t}=-\beta R\nabla E(x), \label{HiSD_intro_x}\\
    \frac{\mathrm{d}R}{\mathrm{d}t}=\gamma(G(x)-RG(x)R).
    \label{HiSD_intro_R}
\end{subnumcases}

Despite wide applications of HiSD in several fields \cite{li2024acta,Yin_Cell2024,Yin_PNAS2021}, 
the convergence of HiSD usually depends heavily on the selection of the initial value. In particular, it was pointed out in \cite{Ortner} that the dimer-type methods may diverge or enter cycles if the initial value is not close enough to the target saddle point. Furthermore, as the solution landscape of complex systems in the literature is usually constructed by employing the HiSD to perform the upward or downward searches from one stationary point to another, which may start from outside of the region of attraction, it is unclear whether all stationary points could be found such that the solution landscape is completely constructed.

To address these issues, we propose an improved high-index saddle dynamics (iHiSD), which integrates the conventional HiSD with gradient flow 
\begin{subnumcases}{\label{iHiSD_intro}}
    \frac{\mathrm{d}x}{\mathrm{d}t}=\beta_1(t)R\nabla E(x)+\beta_2(t)\nabla E(x), \label{iHiSD_intro_x}\\
    \frac{\mathrm{d}R}{\mathrm{d}t}=\gamma(G(x)-RG(x)R). \label{iHiSD_intro_R}
\end{subnumcases}
Here $\beta_1(t)$, $\beta_2(t)$ are time-dependent weight functions and $\gamma$ is a positive parameter. 

The main idea of iHiSD is to integrate the gradient flow in HiSD at the beginning stage since, according to the Morse theory, the gradient flow could connect two saddle points in theory and thus provide possibility of approaching the target saddle point from outside of the region of attraction, which resolves the dependence of the convergence of HiSD on the initial value. Once the trajectory gets close to the target saddle point, the iHiSD switches to the HiSD in order to accurately locate this saddle point. In summary, we expect the iHiSD to overcome the local converge to compute the saddle points and construct the solution landscape by utilizing the advantages of both the gradient flow and the HiSD.

To give a rigorous description, we first propose analysis for the reflection manifold in the iHiSD, and then perform the linear stability analysis of iHiSD. By means of transitive lemma, we prove that the iHiSD exhibits a nonlocal convergence from outside of the region of attraction to the saddle point, which reduces the dependence of the convergence of HiSD on the initial value. A discretized iHiSD is subsequently proposed and analyzed, which inherits these convergence properties. Furthermore,
based on the Morse theory, we prove that any two saddle points could be connected by a sequence of trajectories of iHiSD, which provides possibility of constructing a complete solution landscape (with a finite number of stationary points) through successive upward and downward searches for new saddle points. Different methods are compared in numerical experiments, which indicates the effectiveness of the iHiSD method.

The rest of the paper is organized as follows.
Section 2 introduces preliminary results for future use.
Section 3 performs rigorous analysis for the iHiSD.
Section 4 develops the discrete formulation of iHiSD and proves its convergence.
Section 5 utilizes the Morse theory to demonstrate the effectiveness of iHiSD in constructing solution landscapes. 
Section 6 presents numerical experiments to substantiate the theoretical results. 
A concluding remark is addressed in the last section. 

\section{Preliminaries}
In this section, we present preliminaries to support subsequent analysis.

\paragraph{Morse theory \cite{Morse_Milnor}}
For a function $E$ defined on a Riemannian manifold $\mathcal{M}$,
the {stationary points} are locations on $\mathcal{M}$ where the gradient $\nabla E$ equals to zero.
A stationary point is called {non-degenerate} if the Hessian matrix $G=\nabla\nabla E$ at this point is non-degenerate, and its {(Morse) index} is defined as the negative inertia index of the Hessian.
For brevity, we refer to a stationary point with Morse index $k$ (for $k \geq 0$) as a $k$-saddle.
A function is called a {Morse function} if all  stationary points are non-degenerate, and the
Morse theory is particularly suitable for analyzing the gradient flows of such functions.

\paragraph{Gradient flow, Stable and unstable manifolds \cite{ODE_Teschl}} 
The {gradient flow} of a Morse function $E$ is described by the trajectory followed by 
$$\partial_t \varphi_t(x) = -\nabla E(\varphi_t(x))$$
 with the initial condition $\varphi_0(x) = x$.
%The stationary points, classified by their index, are interconnected through gradient flows.
If $\varphi_t(x)$ represents a trajectory of the gradient flow and 
$$\lim\limits_{t\to -\infty} \varphi_t(x) = x_1,~~\lim\limits_{t\to +\infty} \varphi_t(x) = x_2,$$
 we say there exists a gradient flow from $x_1$ to $x_2$, indicating that the two points are connected via this flow.
%This framework underpins the understanding of the manifold's topology through the {stable and unstable manifolds} associated with these points.

For a general ordinary differential equation (ODE) system $\dot{x} = f(x)$, the stable manifold $W^s(x^*,\varphi)$ (or $W^s(x^*)$ for short) at a fixed point $x^*$ is defined as the set of points $y$ such that $\lim\limits_{t\to +\infty} \varphi_t(y) = x^*$, where $\varphi_t(y)$ denotes the flow of the differential equation starting at $y$.
Conversely, the unstable manifold $W^u(x^*,\varphi)$ (or $W^u(x^*)$ for short) is defined similarly, but for the backward flow: $\lim\limits_{t\to -\infty} \varphi_t(y) = x^*$.
For a hyperbolic fixed point $x^*$, where the Jacobian $Jf(x^*)$ has no eigenvalue with a vanishing real part, the {stable manifold theorem} asserts that the stable and unstable manifolds of $x^*$ are diffeomorphic to Euclidean space, with their dimensions corresponding to the number of eigenvalues of $Jf(x^*)$ with positive and negative real parts, respectively.

\paragraph{Morse-Smale condition \cite{Morse_homology}} 
A Morse function $E$ is called a Morse-Smale function if it satisfies the {Morse-Smale condition}, which requires that the stable and unstable manifolds of any two stationary points intersect transversally.
This property ensures that the gradient flow behaves well, associating the information about stationary points to the geometric properties of the manifold.
It can be shown that for a Morse-Smale function, the gradient flow consistently moves from a stationary point with a higher index to one with a lower index.

\paragraph{Morse's homology theorem \cite{Morse_homology}} 
A fundamental result connecting gradient flow to the topological properties of manifolds is the Morse homology theorem.
It states that for a Morse-Smale function $E$ defined on a compact manifold $\mathcal{M}$, the stationary points of $E$ can be utilized to construct a chain complex whose homology groups are isomorphic to those of $\mathcal{M}$.
Specifically, let \( C^*(E) = \bigoplus_{k \in \mathbb{N}} C_k(E) \), where \( C_k \) is a \( \mathbb{Z}_2 \)-linear space freely generated by the $k$-saddles.
The boundary operator \( \partial_k \) is defined as 
$$\partial_0 = 0,~~ \partial_k: C_k(E) \to C_{k-1}(E) \text{ for } k \geq 1 \text{ with } \partial_k a = \sum_b n(a,b)b,$$
 where the summation runs over all $(k-1)$-saddles and \( n(a,b) \) is the signed count of gradient flows connecting \( a \) and \( b \).
Then the theorem asserts that
$$
H_k(\mathcal{M};\mathbb{Z}_2) \cong H_k(C^*(E)) =\ker\partial_k / \mathrm{im}\ \partial_{k+1},
$$
where \( H_k(\mathcal{M}; \mathbb{Z}_2) \) represents the \( k \)-th homology group of \( \mathcal{M} \) with coefficients in \( \mathbb{Z}_2 \).

If the sublevel set \( \mathcal{M}^c = \{ x : E(x) \leq c \} \) is compact for some \( c \), and there are no stationary points on the level set \( \{ x : E(x) = c \} \), the theorem applies to \( (\mathcal{M}^c, E|_{\mathcal{M}^c}) \). Furthermore, if \( \mathcal{M}^c \) is compact for all \( c \) and contains finitely many stationary points, then \( \mathcal{M} \) is homotopically equivalent to \( \mathcal{M}^c \) for sufficiently large \( c \), and the theorem also holds for \( (\mathcal{M}, E) \).
% where $H_k(\mathcal{M};\mathbb{Z}_2)$ denotes the $k$-th homology group of $\mathcal{M}$ with coefficients in $\mathbb{Z}_2$. 
% If we only assume that the sublevel set $\mathcal{M}^c=\{x:E(x)\leq c\}$ is compact for some $c$ and there are no stationary points on the level set $\{x:E(x)=c\}$, then Morse's homology theorem holds for $(\mathcal{M}^c,E|_{\mathcal{M}^c})$. If we assume that $\mathcal{M}^c$ is compact for all $c$ and there are finite stationary points, then $\mathcal{M}$ is homotopically equivalent to $\mathcal{M}^c$ for sufficiently large $c$, and Morse's homology theorem also holds for $(\mathcal{M},E)$.
~\\

\section{Analysis of iHiSD}
In this section, we consider the stability and nonlocal convergence of the iHiSD \eqref{iHiSD_intro} from a theoretical perspective, where $E:\mathbb{R}^n\to\mathbb{R}$ satisfies the following assumptions:\\
  \textit{Assumption A}:  $E$ is a Morse-Smale function with compact sublevel sets, and its Hessian matrix has distinct eigenvalues at each saddle point.\\
Furthermore, the time-dependent functions $\beta_1(t)$ and $\beta_2(t)$ in (\ref{iHiSD_intro_x}) take the form of
$\beta_1(t)=-\beta\cdot\alpha(t)$ and $\beta_2(t)=\pm\beta\cdot(1-\alpha(t))$, where the sign of $\beta_2(t)$ indicates the search direction: a positive value corresponds to an upward search, while a negative value indicates a downward search. In general, such convex combination of $-R\nabla E(x)$ and $\pm\nabla E(x)$ is not essential.
Without loss of generality, let $\alpha(t)$ be tuned by a rate function $\epsilon(\alpha)$, which, together with \eqref{iHiSD_intro}, leads to the following system
\begin{subnumcases}{\label{iHiSD}}
    \dot{x}=\beta (-\alpha R \pm(1-\alpha)I)\nabla E(x), \label{iHiSD_x}\\
    \dot{R}=\gamma(G(x)-RG(x)R), \label{iHiSD_R}\\
    \dot{\alpha}=\epsilon(\alpha). \label{iHiSD_alpha}
\end{subnumcases}
Here $\beta,\gamma>0$ are hyper-parameters that adjust the time scale, and $G=\nabla\nabla E$ denotes the Hessian matrix of function $E$.
The rate function $\epsilon(\alpha)$ in \eqref{iHiSD_alpha} is assumed smooth and satisfy $\epsilon(\alpha)\sim\alpha(1-\alpha)$, (i.e. there exist $C_1,C_2>0$ such that $C_1\alpha(1-\alpha)\le\epsilon(\alpha)\le C_2\alpha(1-\alpha)$), ensuring that $\alpha$ remains within the interval $[0,1]$.
\subsection{Analysis of reflection manifold}
Compared with the gradient flow,  the reflection matrix $R\in\mathbb{R}^{n\times n}$ in iHiSD introduces salient differences.
Specifically, the matrix $R$ resides within the manifold 
\begin{equation}
\label{Grass_k}
    \mathcal{R}_k = \{R \in \mathrm{Sym}_n : R^2 = I, \dim\ker(R+I) = k\},
\end{equation}
which represents an immersed submanifold of the Grassmannian $\mathrm{Gr}_k(\mathbb{R}^n)$
within the space of symmetric matrices $\mathrm{Sym}_n$ equipped with the Frobenius norm.
To examine basic properties of the reflection manifold $\mathcal{R}_k$, we begin by differentiating the equation $R^2 = I$ to obtain the tangent space, expressed as $$ T_R\mathcal{R}_k = \{A \in \mathrm{Sym}_n : AR + RA = 0\}.$$ 
Furthermore, the orthogonal projection from $\mathrm{Sym}_n$ onto $T_R\mathcal{R}_k$ is given by $A \mapsto P_R A = \frac{1}{2}(A - RAR)$, which is validated by
\begin{equation*}
    \langle P_R A, A - P_R A \rangle = \frac{1}{4}\langle A - RAR, A + RAR \rangle = 0\;\implies P_R A\perp(A - P_R A).
\end{equation*}
Consequently, equations \eqref{HiSD_intro_R} and \eqref{iHiSD_R} can be interpreted as optimizing the Frobenius distance with respect to $R$ through gradient descent
\begin{equation*}
    \dot{R}=-\gamma\nabla_R\|R-G\|_F^2=-\gamma P_R(2(R-G))=\gamma(G-RGR).
\end{equation*}
This means that $R=R^{-1}\in \mathcal{R}_k$ serves as an approximation of the Hessian \( G(x) \). %using a symmetric reflection matrix.
Hence, by substituting $R^{-1} \approx G$ into \eqref{HiSD} to obtain that $\dot{x} \approx -\beta(G(x))^{-1}\nabla E(x)$, the original HiSD can be viewed as an approximation of the Newton method, with the quasi-Newton direction constraining its nonlocal characteristics.
In contrast, iHiSD, as presented in \eqref{iHiSD}, fundamentally transforms this approach by explicitly incorporating gradient flow into the HiSD framework, enhancing convergence properties in both theoretical and practical aspects.

The following lemma further characterizes the qualitative behavior of $\|R-G\|_F^2$.
\begin{lemma}\label{lemma_minimum}
    Given the matrix $G\in \mathrm{Sym}_n$ with eigenvalues $\lambda_1 \leq \cdots \leq \lambda_n$, the function $\|R-G\|_F^2$ defined on $\mathcal{R}_k$ attains a unique local minimum when $\lambda_k < \lambda_{k+1}$ and a unique local maximum when $\lambda_{n-k-1} < \lambda_{n-k}$.
\end{lemma}
\begin{proof}
    The condition for $R$ to be an extremum is that $$\nabla_R\|R-G\|_F^2=0,$$
    or equivalently, $RG=GR,$
    which allows $R$ and $G$ to be orthogonally diagonalized simultaneously
    \begin{equation*}
    \label{GR_simultaneous}
        G = V\begin{bmatrix}\tilde\lambda_1 & & \\ & \ddots & \\ & & \tilde\lambda_n\end{bmatrix}V^\top, \quad
        R = V\begin{bmatrix}e_1 & & \\ & \ddots & \\ & & e_n\end{bmatrix}V^\top.
    \end{equation*}
    Here $\tilde\lambda_1,\cdots,\tilde\lambda_n$ is a permutation of $\lambda_1,\cdots,\lambda_n$.
    Due to the constraint in \eqref{Grass_k},
  we  suppose $e_i=-1$ for $1\le i\le k$ and $e_j=+1$ for $k+1\le j\le n$ without loss of generality.
    Thus, the tangent space $T_R\mathcal{R}_k$ takes the form of
    \begin{equation*}
        T_{R}\mathcal{R}_k=\left\{V\begin{bmatrix}0&B\\B^\top&0\end{bmatrix}V^\top:B\in\mathbb{R}^{k\times(n-k)}\right\}.
    \end{equation*}
    The Hessian $\nabla_R\nabla_R\|R-G\|_F^2$ at the extreme point can now be expressed explicitly by
    \begin{equation*}
        (\delta R\cdot\nabla_R)(\delta R\cdot\nabla_R)\|R-G\|_F^2=\sum_{i=1}^k\sum_{j=k+1}^n\frac12(\tilde\lambda_j-\tilde\lambda_i)\delta R_{ij}^2,
    \end{equation*}
    from which the eigenvalues of $\nabla_R\nabla_R\|R-G\|_F^2$ are found to be $(\tilde\lambda_j-\tilde\lambda_i)$ for $1\le i\le k$ and $k+1\le j\le n$.
    Consequently, $R$ is a local minimum when $\tilde\lambda_i<\tilde\lambda_j$, which corresponds to the condition $\lambda_k<\lambda_{k+1}$. 
    Similarly, $R$ is a local maximum when $\tilde\lambda_i>\tilde\lambda_j$,  or equivalently $\lambda_{n-k-1}<\lambda_{n-k}$. The proof is completed.
\end{proof}

\subsection{Stability and nonlocal convergence}
We present the linear stability of iHiSD.
\begin{theorem}\label{thm_local} 
Suppose the Assumption A holds. The triplet $(x^{*},R^{*},\alpha^{*})$ is an equilibrium point if and only if $\nabla E(x^{*})=0$, $G(x^{*})R^{*}=R^{*}G(x^{*})$ and $\alpha^{*}\in\{0,1\}$.
Moreover, if $G(x^*)$ is non-degenerate and has distinct eigenvalues, then the triplet $(x^*,R^*,\alpha^*)$ is locally stable if and only if $x^*$ is a $k$-saddle, $R^*$ is the reflection about the eigenspaces of $G(x^*)$ corresponding to positive eigenvalues, and $\alpha^*=1$.
\end{theorem}

\begin{proof}
For the proof of sufficiency, the condition $\nabla E(x^*)=0$ implies that the right-hand side of \eqref{iHiSD_x} vanishes. 
Additionally, the relation $G(x^{*})R^{*}=R^{*}G(x^{*})$ combined with the constraint $(R^{*})^2=I$ ensures that the right-hand side of \eqref{iHiSD_R} also vanishes. 
Furthermore, the assumption $\epsilon(\alpha)\sim\alpha(1-\alpha)$ leads to $\epsilon(\alpha^*)=0$ for $\alpha^{*}\in\{0,1\}$.

Concerning necessity, the assumption on $\epsilon(\alpha)$ implies that $\alpha^{*}\in\{0,1\}$ at equilibrium.
In both cases, $-\alpha^* R^*\pm(1-\alpha^*)I$ is non-degenerate, and thus $\dot{x}=0$ implies $\nabla E(x^*)=0$.
The constraint $(R^*)^2=I$ together with $\dot{R}=0$ imply $G(x^*)R^*=R^*G(x^*)$.
Furthermore, the assumption $\epsilon(\alpha)\sim\alpha(1-\alpha)$ implies $\alpha^*=1$.
The Jacobian of the iHiSD vector field at $(x^*,R^*,\alpha^*)$ can be computed as
\begin{equation*}
    J(x^*,R^*,1)(\delta x,\delta R,\delta\alpha)=
    \begin{pmatrix}
        -\beta R^*G(x^*)\delta x\\
        \gamma(\mathcal{A}^*\delta x-R^*(\mathcal{A}^*\delta x)R^*)-\gamma\nabla_R\nabla_R\|R-G(x^*)\|_F^2\cdot\delta R\\
        \epsilon'(1)\delta\alpha
    \end{pmatrix},
\end{equation*}
where $\mathcal{A}=\nabla G$.
By orthogonally diagonalizing $G(x^*)$ and $R^*$ simultaneously just as that in Lemma \ref{lemma_minimum},
the eigenvalues of \(J(x^*,R^*,\alpha^*)\) could be figured out as $\beta\tilde\lambda_i,\;-\beta\tilde\lambda_j,$ $\gamma(\tilde\lambda_i-\tilde\lambda_j)$ and $\epsilon^\prime(1)$, respectively,
for $1\leq i\leq k$ and $k+1\leq j\leq n$.
For non-degenerate cases where $\lambda_i\neq0$ and $\lambda_i\neq\lambda_j$ for any $i\neq j$,
 $(x^*,R^*,\alpha^*)$ is a stable equilibrium point if and only if $\tilde\lambda_i<0<\tilde\lambda_j$ for $1\leq i\leq k$ and $k+1\leq j\leq n$.
In other words, $x^*$ is a $k$-saddle,
and $\ker(R^*+I)$ corresponds to the eigenspaces of $G(x^*)$ with negative eigenvalues,
i.e., $R^*$ is the reflection about the eigenspaces of $G(x^*)$ with positive eigenvalues. The proof is thus completed.
\end{proof}

We now turn to the analysis of nonlocal convergence, which indicates that the convergence requirement on initial conditions could be relaxed, e.g. it is not necessary for the initial point to locate in the region of attraction of the target saddle point. 
The theoretical analysis relies on the following transitive lemma.%, which could be derived from the $\lambda$-Lemma, see e.g. the Chapter 6 in \cite{banyaga2004lectures}.  Nevertheless, we propose a more concise and straightforward proof without using .

\begin{lemma}[transitive lemma]
\label{lemma_lambda}
Suppose $\dot{x}=f(x)$ is a dynamical system on the manifold $\mathcal{M}$, and $p\in\mathcal{M}$ is a hyperbolic stationary point.
$\mathcal{M}_1,\mathcal{M}_2$ are invariant sub-manifolds.
Suppose $\mathcal{M}_1$ meets $W^s(p)$ transversally at $q$,
and $\mathcal{M}_2$ meets $W^u(p)$ transversally at $r$.
Then, $\mathcal{M}_1\cap\mathcal{M}_2\neq\emptyset$.
\end{lemma}
\begin{proof}
Let $(U,X)$ be a chart near $p$ with the coordinates $X$ decomposed as
$X = (X^s, X^u)$ such that 
$$
W^u(p) \cap U = \{X \in U : X^s = 0\},\;W^s(p) \cap U = \{X \in U : X^u = 0\}.
$$
Thus, the Jacobian matrix at point $p$ can be expressed blockwise as 
$$
A = Jf(p) = \begin{bmatrix} A_s & 0 \\ 0 & A_u \end{bmatrix},
$$
where $A_s$ and $A_u$ are matrices whose spectrums
lie in the left and right half plane, respectively.
Correspondingly, the flow $\varphi_t(X)$ of $\dot{x}=f(x)$ can be decomposed into the linear term and the residual term
$$
\varphi_t(X)=\left(\exp(A_st)X^s+\eta^s_t(X),\;\exp(A_ut)X^u+\eta^u_t(X)\right),
$$
where the residual 
$\eta_t(X) = (\eta^s_t(X), \eta^u_t(X))$
satisfies $\nabla_X \eta_t(0) = 0$.

Without loss of generality, we can assume that $q\in U$, otherwise we can replace it with $\varphi_T(q)$ which is close to $p$ with large $T$, while keeping the transversality condition between $\mathcal{M}_1$ and $W^s(p)$.
Since $T_qW^s(p)=\operatorname{span}\left\{\frac{\partial}{\partial X^s}\right\}$ and $T_q\mathcal{M}=T_q\mathcal{M}_1+T_qW^s(p)$,
we can decompose the coordinates $X^s$ into $X^s=(X^{s,1},X^{s,2})$ such that
$$
T_q\mathcal{M}=T_q\mathcal{M}_1\oplus\operatorname{span}\left\{\frac{\partial}{\partial X^{s,1}}\right\}.
$$
The implicit function theorem implies that there exists some neighborhood of $q$, denoted by $U_q$, such that
\begin{equation}
\label{def:g}
    \mathcal{M}_1\cap U_q=\{X\in U_q: X^{s,1}=g(X^{s,2},X^u)\}
\end{equation}
with $g\in\mathcal{C}^1.$
Similarly, we can decompose $X^u=(X^{u,1},X^{u,2})$ such that
$$T_r\mathcal{M}=T_r\mathcal{M}_2\oplus\operatorname{span}\left\{\frac{\partial}{\partial X^{u,2}}\right\}$$ and that at the neighborhood of $r$
\begin{equation}
\label{def:h}
    \mathcal{M}_2\cap V_r=\{X\in V_r: X^{u,2}=h(X^s,X^{u,1})\}
\end{equation}
where $h\in\mathcal{C}^1.$

We shall prove $\mathcal{M}_1\cap\mathcal{M}_2\neq\emptyset$ by finding a point $a\in\mathcal{M}_1$ such that $\varphi_t(a)\in\mathcal{M}_2$ for some $t\in\mathbb{R}$. Once we achieve that, and by noticing that $\mathcal{M}_1$ is invariant under the flow $\varphi$, we can conclude from $a\in\mathcal{M}_1$ that $\varphi_t(a)\in\mathcal{M}_1$, and consequently $\mathcal{M}_1\cap\mathcal{M}_2\neq\emptyset$.

Let us conduct a rigorous examination of the point $a$ defined by
\begin{equation}
\label{Ft}
    a = \mathcal{F}_t(Y^{s,1},Y^{u,2}) := \left(Y^{s,1}, X^{s,2}(q), \exp(-A_ut)(X^{u,1}(r), Y^{u,2})\right).
\end{equation}
Here, $Y^{s,1}$ and $Y^{u,2}$ should be picked carefully such that the point $a$ meets our requirement.
By definition \eqref{def:g}, $a\in\mathcal{M}_1$ is equivalent to
\begin{equation}
\label{Ys}
    Y^{s,1}=g\left(X^{s,2}(q),\exp(-A_ut)\left(X^{u,1}(r),Y^{u,2}\right)\right).
\end{equation}
Similarly, by definition \eqref{def:h}, $\varphi_t(a)\in\mathcal{M}_2$ is equivalent to
\begin{equation}
\label{Yu}
    Y^{u,2}+\eta^{u,2}_t(a)=h\left(\exp(A_st)\left(Y^{s,1},X^{s,2}(q)\right)+\eta^s_t(a),X^{u,1}(r)+\eta^{u,1}_t(a)\right).
\end{equation}
The equations \eqref{Ys}--\eqref{Yu} are equivalent to the fix point of the following mapping $\Psi_t\in\mathcal{C}^1$
\begin{align*}
&\Psi_t(Y^{s,1},Y^{u,2})=\\
&\begin{pmatrix}
g\left(X^{s,2}(q),\exp(-A_ut)\left(X^{u,1}(r),Y^{u,2}\right)\right)\\
-\eta^{u,2}_t(a)+h\left(\exp(A_st)\left(Y^{s,1},X^{s,2}(q)\right)+\eta^s_t(a),X^{u,1}(r)+\eta^{u,1}_t(a)\right)
\end{pmatrix}.
\end{align*}
We shall prove its existence
by showing that $\Psi_t$ is a contraction mapping.
First, it can be checked by definition that $\Psi_t$ is locally Lipschitz continuous with
\begin{align*}
\|\Psi_t\|_{\mathrm{Lip}}\leq&\|g\|_{\mathrm{Lip}}\|\exp(-A_ut)\|+\|\eta_t\|_{\mathrm{Lip}}\|\mathcal{F}_t\|_{\mathrm{Lip}}\\
&\qquad+\|h\|_{\mathrm{Lip}}\left(\|\exp(A_st)\|+\|\eta_t\|_{\mathrm{Lip}}\|\mathcal{F}_t\|_{\mathrm{Lip}}\right),
\end{align*}
where the notation $||\cdot||_{\text{Lip}}$ refers to the Lipschitz continuity constant.
Since $\nabla_X\eta_t(0) = 0$, we can shrink the domain of interest $U$ such that $\|\eta_t\|_{\mathrm{Lip}(U)}$ is almost zero.
Furthermore, by definition \eqref{Ft} we have
$\|\mathcal{F}\|_{\mathrm{Lip}}\leq 1+\|\exp(-A_u t)\|$,
and thus for sufficiently large $t$,
we can make $\|\Psi_t\|_{\mathrm{Lip}(U)}$ almost zero.

In addition, a sufficiently large $t$ also ensures that $\Psi_t$ maps the neighborhood of $(X^{s,1}(q),X^{u,2}(r))$ into itself. This can be verified through
\begin{align*}
    \lim_{t\to\infty}\mathcal{F}_t(X^{s,1}(q),X^{u,2}(r))&=q,\\
    \lim_{t\to\infty}\Psi_t(X^{s,1}(q),X^{u,2}(r))&=(X^{s,1}(q),X^{u,2}(r)).
\end{align*}

Finally,
by restricting on a sufficiently small domain $U$,
and by choosing a sufficiently large $t$,
we have made $\Psi_t$ a contraction mapping on the neighborhood of $(X^{s,1}(q),X^{u,2}(r))$.
Take its fixed point and we have obtained a solution of equations \eqref{Ys} and \eqref{Yu},
which results in the expected element $a \in \mathcal{M}_1$ with $\varphi_t(a) \in \mathcal{M}_2$. The proof is thus completed.
\end{proof}

With the assistance of the transitive lemma, we present the nonlocal convergence result in the following theorem that could resolve the dependence of the convergence of HiSD on the initial value.

\begin{theorem}[nonlocal convergence]\label{thm_nonlocal}
Consider two stationary points $x_1$ and $x_2$ of $E$, which satisfies the Assumption A, and suppose $x_2$ has index $k$. 
If a path of $\dot{x} = \pm \nabla E(x)$ (following the sign convention in  \eqref{iHiSD_x}) exists from $x_1$ to $x_2$, then there is a solution $(x(t), R(t), \alpha(t))$ of iHiSD \eqref{iHiSD} that approximates this path, with 
$\lim\limits_{t \to -\infty} x(t) = x_1$ and $\lim\limits_{t \to +\infty} x(t) = x_2.$
\end{theorem}
\begin{remark}
This theorem assumes the existence of a gradient flow between stationary points, which will be justified when constructing the solution landscapes, cf the proof of Theorem \ref{thm_complete}. 
\end{remark}
% With this lemma in hand, we shall begin the proof of theorem \ref{thm_nonlocal}.

\begin{proof}
According to Lemma \ref{lemma_minimum}, denote $R_1$ as the unique local maximum of $\|R-G(x_1)\|_F^2$ and $R_2$ as the unique local minimum of $\|R-G(x_2)\|_F^2$.

The proof consists of two main parts: we will show that for the iHiSD system \eqref{iHiSD}, there exists a trajectory from $q=(x_1,R_1,0)$ to $p=(x_2,R_2,0)$, and another trajectory from $p=(x_2,R_2,0)$ to $r=(x_2,R_2,1)$, along with certain transversality conditions. 
Consequently, Lemma \ref{lemma_lambda} implies the existence of a trajectory from $q=(x_1,R_1,0)$ to $r=(x_2,R_2,1)$, which is precisely what we aim to establish.

First, we consider the case that $\alpha(t)\equiv0$, resulting in the reduced system
\begin{equation}
\label{decouple}
\left\{
\begin{aligned}
\dot x&=\pm\beta\nabla E(x),\\
\dot R&=\gamma(G(x)-RG(x)R).\end{aligned}
\right.
\end{equation}
We assert that there exists a path of \eqref{decouple} such that
\begin{equation*}
    \lim_{t\to-\infty}(x(t),R(t))=(x_1,R_1),\
    \lim_{t\to+\infty}(x(t),R(t))=(x_2,R_2).
\end{equation*}
Indeed, we can choose $x(t)$ to be the gradient flow connecting $x_1$ and $x_2$, which forms a one-dimensional manifold. 
The unique stable equilibrium point of \eqref{decouple} is $(x_2, R_2)$, while the unique reversibly stable equilibrium point is $(x_1, R_1)$. 
According to Sard's theorem \cite{Sard1942TheMO}, the union of the stable manifolds of equilibrium points, excluding $(x_2, R_2)$, and the unstable manifolds of equilibrium points, excluding $(x_1, R_1)$, constitutes a set of measure zero. 
Therefore, there exists a corresponding solution $R(t)$ such that $\lim_{t \to -\infty} R(t) = R_1$ and $\lim_{t \to +\infty} R(t) = R_2$. 

Denote the flow generated by iHiSD \eqref{iHiSD} as $\varphi$; thus, we have proved that
\begin{equation}
\label{trans_cond}
W^s((x_2,R_2,0),\varphi)\cap W^u((x_1,R_1,0),\varphi)\neq\emptyset.
\end{equation}

Next comes the transversality condition.
We assert that
\begin{equation}
\label{tangent_space_1}
    T_{(x_0,R_0,0)}W^s((x_2,R_2,0),\varphi)=T_{x_0}W^s(x_2)\times T_{R_0}\mathcal{R}_k\times\{0\},
\end{equation}
where $W^s(x_2)$ denotes the stable manifold of $x_2$ corresponding to $\dot x=\pm\beta\nabla E(x)$.
% $$T_{(x_0,R_0,0)}W^u((x_1,R_1,0),\varphi)=T_{x_0}W^u(x_1)\times T_{R_0}\mathcal{R}_k\times\mathbb{R}.$$
% Here, $W^s(x_2),W^u(x_1)$ denotes the stable and unstable manifolds corresponding to 

For the inclusion part, consider any path $(x(t),R(t),\alpha(t))$ of iHiSD \eqref{iHiSD} that lies on the stable manifold, meaning that $\lim\limits_{t\to+\infty}x(t)=x_2$, $\lim\limits_{t\to+\infty}\alpha(t)=0$.
The assumption on $\epsilon(\alpha)$ implies that $\alpha(t) \equiv 0$, yielding $\dot{x} = \pm \beta \nabla E(x)$ in \eqref{iHiSD_x}. Thus, $x(t) \in W^s(x_2)$, and hence $$T_{(x_0, R_0, 0)} W^s((x_2, R_2, 0), \varphi) \subseteq T_{x_0} W^s(x_2) \times T_{R_0} \mathcal{R}_k \times \{0\}.$$

% Conversely

% On the one hand, $T_{(x_0,R_0,0)}W^s((x_2,R_2,0),\varphi)\subset T_{x_0}W^s(x_2)\times T_{R_0}\mathcal{R}_k\times\{0\}$. This is implicit in the fact that $W^s((x_2,R_2,0),\varphi)\subset W^s(x_2)\times\mathcal{R}_k\times\{0\}$. In fact, assume that $(x(t),R(t),\alpha(t))$ is a path of iHiSD (\ref{iHiSD}) such that 

The converse part relies on the dimensional equality.
According to the proof of Theorem \ref{thm_local} and Lemma \ref{lemma_minimum}, $(x_2,R_2,0)$ has $(k+1)$ unstable directions and thus
\begin{align*}
    &\dim T_{(x_0,R_0,0)}W^s((x_2,R_2,0),\varphi)\\
    &\qquad=(n+k(n-k)+1)-(k+1)=(n-k)+k(n-k)\\
    &\qquad=\dim (T_{x_0}W^s(x_2)\times T_{R_0}\mathcal{R}_k\times\{0\}).
\end{align*}
Therefore, \eqref{tangent_space_1} has been verified. Similarly, we have
\begin{equation*}
T_{(x_0,R_0,0)}W^u((x_1,R_1,0),\varphi)=T_{x_0}W^u(x_1)\times T_{R_0}\mathcal{R}_k\times\mathbb{R}.
\end{equation*}
By the Morse-Smale condition, it follows the transversality condition of \eqref{trans_cond}
\begin{align*}
&T_{(x_0,R_0,0)}W^s((x_2,R_2,0),\varphi)+T_{(x_0,R_0,0)}W^u((x_1,R_1,0),\varphi)\\
&\qquad=(T_{x_0}W^u(x_1)+T_{x_0}W^s(x_2))\times T_{R_0}\mathcal{R}_k\times\mathbb{R}=\mathbb{R}^n\times T_{R_0}\mathcal{R}_k\times\mathbb{R}.
\end{align*}

Now we consider another case when $x(t)\equiv x_2$ and $R(t)\equiv R_2$.
By the condition on $\epsilon$ in \eqref{iHiSD_alpha}, it is obvious that
\begin{equation*}
\forall \alpha_0\in(0,1),\;\;
    \lim_{t\to-\infty}\varphi_t(x_2,R_2,\alpha_0)=(x_2,R_2,0),\
    \lim_{t\to+\infty}\varphi_t(x_2,R_2,\alpha_0)=(x_2,R_2,1).
\end{equation*}
That is, $W^u((x_2,R_2,0),\varphi)\cap W^s((x_2,R_2,1),\varphi)\neq\emptyset$. The transversality condition holds because $(x_2,R_2,1)$ is a stable equilibrium point proved by theorem \ref{thm_local}.

To sum up, the condition of Lemma \ref{lemma_lambda} holds for $\mathcal{M}_1=W^u((x_1,R_1,0),\varphi)$, $\mathcal{M}_2=W^s((x_2,R_2,1),\varphi)$, $p=(x_2,R_2,0)$
and we have proved the existence of iHiSD \eqref{iHiSD} trajectory from $(x_1,R_1,0)$ to $(x_2,R_2,1)$.
\end{proof}

It is important to note that the nonlocal convergence property does not apply to HiSD, which is quasi-Newton, as will be demonstrated numerically in Section \ref{sec:butterfly}. 
In contrast, this property is confirmed for iHiSD, as indicated by Theorem \ref{thm_nonlocal}, and is crucial for constructing the solution landscape. 

\begin{remark}
Theorem \ref{thm_nonlocal} is based on the existence of a gradient flow between the stationary points \(x_1\) and \(x_2\), which requires that their Morse indices should be distinct. Nevertheless, it is not required for these indices to differ by $\pm1$, a typical scenario in the Morse theory. 
When $x_1$ has a smaller index, the sign in $\pm$ in \eqref{iHiSD_x} should be $+$ for upward search, and vice versa. 
\end{remark}

\section{Discretization of iHiSD}
We consider the numerical discretization for iHiSD for the sake of implementation. In the last section,
Theorem \ref{thm_local} implies the stability of iHiSD \eqref{iHiSD} in locating a $k$-saddle. 
However, directly applying a general ODE solver is impractical, as the constraint $R\in\mathcal{R}_k$ may be violated during the iterations. 
Fortunately, since $\dim\ker(R+I)=k\le n$ (with $k\ll n$ in typical scenarios), we can instead track the orthonormal basis $\{v_i\}_{i=1}^k$ of the low-dimensional subspace $\ker(R+I)$, from which we can represent $R\in\mathcal{R}_k$ equivalently as $R=I-2\sum_{i=1}^kv_iv_i^T$, with $v_i^Tv_j=\delta_{ij}$.

Additionally, it has been shown that the formulation in \eqref{iHiSD_R} corresponds to optimizing $\|R-G\|_F^2$ via gradient descent. 
By expressing $R=I-2\sum_{i=1}^kv_iv_i^T$, the optimization problem can be reformulated as
\begin{equation}
    \min_{v_1,\cdots,v_k}\sum_{i=1}^k v_i^T G v_i,
    \quad\mathrm{s.t.}\;~v_i^Tv_j=\delta_{ij},
\end{equation}
which seeks the eigenspaces corresponding to the smallest $k$ eigenvalues of the Hessian $G(x)$, denoted as $\{v_i\}_{i=1}^k=\text{EigenSolve}_k\left(G\right)$ in the following discussion.
To solve this, we can employ methods such as the Simultaneous Rayleigh-Quotient Iterative Minimization \cite{RayleighQuotient} or the Locally Optimal Block Preconditioned Conjugate Gradient method \cite{LOBPCG}. Based on these approaches, the forward Euler discretization of iHiSD \eqref{iHiSD} is outlined in Algorithm \ref{alg_iHiSD}.

\begin{algorithm}[htbp]
\caption{iHiSD method for a $k$-saddle}
\label{alg_iHiSD}
\begin{algorithmic}
\Require
initial state $x_0$,
search direction $s\in\{\pm1\}$,
target index $k\in\mathbb{N}$,
ratio sequence $\alpha=(\alpha_m)_{m=0}^{\infty}\in\mathbb{R}^\mathbb{N}$,
tolerance $\varepsilon_F$.
\Ensure $k$-saddle $x^*$.
\State $m\leftarrow0,\;g_0\leftarrow\nabla E(x_0),\;G_0\leftarrow\nabla\nabla E(x_0)$.
\Repeat
\State $\left\{v_{m,i}\right\}_{i=1}^k\leftarrow\text{EigenSolve}_k\left(G_m\right)$.
% \State $\beta_1^{(m)}\leftarrow\beta\cdot\alpha^{(m)}$
% \State $\beta_2^{(m)}\leftarrow \beta\cdot(1-\alpha^{(m)})\cdot s$
\State $d_m\leftarrow((1-\alpha_m)s-\alpha_m)\cdot g_m+2\alpha_m\sum_{i=1}^k\langle v_{m,i},g_m\rangle v_{m,i}$.
\State $x_{m+1}\leftarrow x_m+\eta_m\cdot d_m$.
\State $g_{m+1}\leftarrow\nabla E(x_{m+1}),\;G_{m+1}\leftarrow\nabla\nabla E(x_{m+1})$.
% \State $\alpha^{(m+1)}\leftarrow\alpha^{(m)}+\delta t\cdot\alpha^{(m)}(1-\alpha^{(m)})$
\State $m\leftarrow m+1$.
\Until $\|g_m\|<\varepsilon_F$
\end{algorithmic}
\end{algorithm}

To establish the convergence of Algorithm \ref{alg_iHiSD}, we impose the following  settings similar to those used in optimization theory. 
For a symmetric matrix $A$, let $|A|$ represent the square root of $A^2$.

\begin{assumption}
    $x^*$ is a $k$-saddle, and $0<\mu\le\lambda_{\min}(|G(x^*)|)\le\lambda_{\max}(|G(x^*)|)\le L$.
\end{assumption}
\begin{assumption}
    $G(x)$ is $M$-Lipschitz continuous in the neighborhood $U(x^*,\delta)$, meaning that
    $$\|G(x)-G(y)\|_2\leq M\|x-y\|_2,~\forall x,y\in U(x^*,\delta).$$
\end{assumption}

With these assumptions, we can present the following theorem. 
\begin{theorem}
\label{localconv_discrete}
   Suppose the Assumptions 1-2 hold. Assume that after $m_0$ iterations, the condition \( \alpha_m \ge \frac{1 + \varepsilon}{2} > \frac{1}{2} \) holds for \(m > m_0\) and some $0<\varepsilon<1$, and  $$ r_{m_0} = \|x_{m_0} - x^*\|_2 < \min\left(\delta, \frac{2\mu\varepsilon}{3M}\right) .$$ Then, the  step size condition \( \eta_m = \frac{2}{L + \mu(2\alpha_m - 1)} \) ensures that the Algorithm \ref{alg_iHiSD} converges to \(x^*\) with a linear convergence rate of \( \frac{\kappa + \varepsilon}{\kappa + 3\varepsilon} \), where $\kappa=L/\mu$ denotes the condition number.
\end{theorem}
\begin{remark}
Note that as $\alpha\rightarrow 1^-$, which means that $\alpha$ tends to $1$ such that the iHiSD approaches the HiSD, the convergence rate proved in this theorem tends to $\frac{\kappa+1}{\kappa+3}$, which is exactly the convergence rate for HiSD proved in \cite[Theorem 5.2]{Luosiam2022}. Thus, the convergence rate proved in the above theorem consistently generalizes the result in \cite[Theorem 5.2]{Luosiam2022}.
\end{remark}
\begin{proof} 
Introduce the matrices as follows
$$R_m = I - 2\sum_{i=1}^k v_{m,i} v_{m,i}^T,\quad
H_m = \alpha_m R_m + (1 - \alpha_m) sI.$$ 
Here, $s\in\{\pm1\}$ represents the search direction in Algorithm \ref{alg_iHiSD}. 

To establish the convergence, we firstly prove by induction that the sequence $\{r_m := \|x_m - x^*\|_2\}$ is non-increasing such that the local Lipschitz condition could be applied. 
We begin our analysis by examining the recursive relation
\begin{align*}
    x_{m+1}-x^*&=x_m-x^*-\eta_m H_m(g_m-g^*) \qquad\text{(since $g^*=\nabla E(x^*)=0$)}\\
    &=x_m-x^*-\eta_m H_m \int_0^1 G(x^*+t(x_m-x^*))\cdot(x_m-x^*)dt\\
    &=(I-\eta_m H_m G_m)\cdot(x_m-x^*)\\
    &\qquad-\eta_m H_m\int_0^1[G(x^*+t(x_m-x^*))-G_m]dt\cdot(x_m-x^*).
\end{align*}
The induction hypothesis implies $x_m \in U(x^*, \delta)$, ensuring the Lipschitz continuity of $G(x)$. 
Given that $\|H_m\|_2\le\alpha_m\|R_m\|_2+(1-\alpha_m)\|sI\|_2=1$, we can derive
\begin{equation}
\label{contraction}
    r_{m+1}\le\|I-\eta_m H_m G_m\|_2\cdot r_m+\frac12\eta_mM\cdot r_m^2.
\end{equation}
Using the induction hypothesis we can deduce that $r_m\le r_{m_0}<\frac{2\mu\varepsilon}{3M}$, which leads to 
\begin{equation*}
    |\lambda_j(G_m)-\lambda_j(G(x^*))|\le\|G_m-G(x^*)\|_2\le Mr_m<\mu,~\forall j.
\end{equation*}
Thus, $G_m$ has exactly $k$ negative eigenvalues, implying that $R_mG_m=|G_m|$. 
Furthermore, the eigenvalue $ |\lambda(G_m)| $ can be bounded as
$$
\mu - Mr_m \leq |\lambda(G_m)| \leq L + Mr_m.
$$
Consequently, the eigenvalues of $ H_mG_m=\alpha_m|G_m|+(1-\alpha_m)sG_m$ have the following estimation
\begin{align*}
    \lambda(H_mG_m)&=\alpha_m|\lambda(G_m)|+(1-\alpha_m)s\lambda(G_m)\le|\lambda(G_m)|\le L+Mr_m,\\
    \lambda(H_mG_m)&=(1-\alpha_m)(|\lambda(G_m)|+s\lambda(G_m))+(2\alpha_m-1)|\lambda(G_m)|\\&\qquad\ge(2\alpha_m-1)(\mu-Mr_m).
\end{align*}
Therefore, the recursive inequality relationship in \eqref{contraction} leads to
\begin{align*}
    r_{m+1}&\le\max\{|1-\eta_m(L+Mr_m)|,|1-\eta_m(2\alpha_m-1)(\mu-Mr_m)|\}\cdot r_m+\frac12\eta_mM\cdot r_m^2\\
    &\le\max\{|1-\eta_mL|,|1-\eta_m(2\alpha_m-1)\mu|\}\cdot r_m+\frac32\eta_mM\cdot r_m^2.
\end{align*}
By choosing the optimal step size $\eta_m=2/(L+(2\alpha_m-1)\mu)$, it turns out that
\begin{align*}
\label{contraction2}
    r_{m+1}\le\frac{\kappa-(2\alpha_m-1)}{\kappa+(2\alpha_m-1)}\cdot r_m+\frac{3M}{\mu[\kappa+(2\alpha_m-1)]}\cdot r_m^2.
\end{align*}
Since $\alpha_m\ge\frac{1+\varepsilon}{2}$, we have 
$$
r_{m+1}\le\frac{\kappa-\varepsilon}{\kappa+\varepsilon}r_m+\frac{3M}{\mu(\kappa+\varepsilon)}r_m^2.
$$
Combining this with the induction hypothesis $r_m\le r_{m_0}<\frac{2\mu\varepsilon}{3M}$ ensures that $r_{m+1}\le r_m$, thereby completing the inductive procedure.

With the assistance of the above result, the linear convergence result is based on the following property, whose details can be found in \cite{ConvexOpt}
\begin{equation*}
\begin{cases}
    0\le r_{m+1}\le(1-q)r_m+cr_m^2\\
    0\le r_{m_0}<q/c
\end{cases}\implies
r_m\le\min\left\{r_{m_0},\frac{(1+q)^{-(m-m_0)}q r_{m_0}}{q - c r_{m_0}}\right\}.
\end{equation*}
Taking $q=\frac{2\varepsilon}{\kappa+\varepsilon}, c=\frac{3M}{\mu(\kappa+\varepsilon)}$, we can conclude that when $r_{m_0}<\min(\delta,\frac{2\mu\varepsilon}{3M})$,
\begin{equation*}
    \|x_m-x^*\|_2\le\left(\frac{\kappa+\varepsilon}{\kappa+3\varepsilon}\right)^{m-m_0}\cdot\frac{\|x_{m_0}-x^*\|_2}{1-\frac{3M}{2\mu\varepsilon}\|x_{m_0}-x^*\|_2},\quad\forall m>m_0,
\end{equation*}
which completes the proof.
\end{proof}

\section{Constructing solution landscape by iHiSD}
We intend to construct the solution landscape based on the Algorithm \ref{alg_iHiSD}. Despite growing applications of the solution landscape in various fields, whether a solution landscape with a finite number of stationary points could be completely constructed by certain algorithms, a critical issue in providing the comprehensive description for complex systems, still remains unclear. With the help of iHiSD method, we could now partly answer this question (in Remark \ref{zzz}) based on the following theorem. 
\begin{theorem}
\label{thm_complete}
Given any two stationary points $\hat{x}$ and $y$ of $E$ satisfying the Assumption A, there exists a sequence of stationary points $x_0, x_1, \ldots, x_n$, such that $x_0 = \hat{x}$, $x_n = y$, and each consecutive pair $x_{i-1}, x_i$ is connected via the iHiSD method.
\end{theorem}

\begin{remark}\label{zzz}
 Theorem \ref{thm_complete} indicates that starting from any known stationary point, one could apply the Algorithm \ref{alg_iHiSD} to thoroughly explore and connect other stationary points such that the solution landscape with a finite number of stationary points could be completely constructed in theory.
\end{remark}

\begin{proof}
The proof of Theorem \ref{thm_complete} could be divided into two parts.
First, we prove that any stationary point $p$ can be connected to some local minimum $q$ via gradient flow.
Indeed, since the sublevel set $\{x:E(x)\le c\}$ is compact for any $c\in\mathbb{R}$, we have the decomposition $\mathbb{R}^n=\bigcup_q W^s(q)$, where the union is taken over all stationary points of $E$.
By transverlity condition, it holds that 
$$
\dim W^u(p)\cap W^s(q)=\dim W^u(p)+\dim W^s(q)-n=\operatorname{Index}(p)-\operatorname{Index}(q).
$$
Applying Sard's theorem \cite{Sard1942TheMO}, we can infer that $\bigcup_{\operatorname{Index}(q)\geq1}W^u(p)\cap W^s(q)$ is a set of zero measure in $W^u(p)$. 
This implies the existence of a stationary point $q$ with index 0 (i.e., a local minimum) such that
$$
W^u(p)\cap W^s(q)\neq\emptyset,
$$
which, together with the definitions of $W^u(p)$ and $ W^s(q)$, in turn implies that there exists a gradient flow from $p$ to $q$.

Next, we prove that any two local minima $a, b$ can be connected via a series of gradient flows between stationary points arranged alternately with 1-saddles and 0-saddles.
Morse homology theorem states that $\ker\partial_0/\mathrm{im}\partial_1\cong\mathbb{Z}_2$,
and thus $\dim_{\mathbb{Z}_2}\ker\partial_0=1+\dim_{\mathbb{Z}_2}\mathrm{im}\partial_1$.
% The $\ker\partial_0=C_0$ by $\partial_0=0$.
Since $W^u(q)$ is one-dimensional for any $1$-saddle $q$, it must connects to exactly two (maybe the same) local minima, i.e., $\partial_1q=0$ or $r+r'$.
Take the sum to obtain
$\mathrm{im}\partial_1\subset\left\{\sum c_i z_i\in C_0:\sum c_i=0\right\}$ and
$$\dim_{\mathbb{Z}_2}\mathrm{im}\partial_1\leq\dim_{\mathbb{Z}_2}\left\{\sum c_i z_i\in C_0:\sum c_i=0\right\}=\dim_{\mathbb{Z}_2} C_0-1=\dim_{\mathbb{Z}_2}\mathrm{im}\partial_1.$$
This implies that
$$\mathrm{im}\partial_1=\left\{\sum c_i z_i\in C_0:\sum c_i=0\right\},$$
and hence $a+b\in\mathrm{im}\partial_1$, i.e., $a+b=\partial_1\sum_{i=1}^kq_i$.
By definition of $\partial_1$, we have found the desired series of $1$-saddles,
as well as the $0$-saddles between them.

Overall, we can connect any two stationary points, by first connecting them to the local minima, and then connecting with alternative $1$-saddles and minima.
\end{proof}
\begin{remark}
Although the proof mainly relies on the local minima and the $1$-saddles, in practice it is more efficient to construct the solution landscape based on saddle points with larger indices since they have unstable manifolds with higher dimension, and thus might be connected to more saddle points by iHiSD.
\end{remark}

With the assistance of Theorem \ref{thm_complete}, we could iteratively apply Algorithm \ref{alg_iHiSD} starting from known stationary points to discover additional stationary points in order to systematically construct the solution landscape, as outlined in the following algorithm.
\begin{algorithm}[H]
\caption{iHiSD framework for constructing solution landscape}
\label{alg_sl}
\begin{algorithmic}
\Require A known stationary point $\hat{x}$ (potentially a minimum obtained via gradient descent).
\Ensure The solution landscape $\mathcal{G}=(\mathcal{V},\mathcal{E})$.
\State Initialize the vertex set: $\mathcal{V} \leftarrow \{\hat{x}\}$.
\State Initialize the edge set: $\mathcal{E} \leftarrow \emptyset$.
\Repeat
\State Select a stationary point $x\in \mathcal{V}$ with index $l$ and perturb it to obtain $x_\delta$.
\State Specify the target index $k\neq l$.
\State Set the direction variable $s=+1$ if $l < k$, otherwise $s=-1$.
\State Choose an initial ratio $\alpha_0\gtrsim0$.
\State Execute Algorithm \ref{alg_iHiSD} with initial value $(x_\delta,\alpha_0)$, terminating at a $k$-saddle $y$.
\State Update the vertex set: $\mathcal{V}\leftarrow \mathcal{V}\cup\{y\}$.
\State Update the edge set: $\mathcal{E}\leftarrow \mathcal{E}\cup\{(x,y)\}$.
\Until No new stationary points are found
\end{algorithmic}
\end{algorithm}

Finally, it is worth mentioning that although the Algorithm $\ref{alg_sl}$ may require an extensive and thorough search between stationary points with different Morse indices, it often suffices to limit the search between adjacent indices. The reason is that for a $k$-saddle $x$ ($k \geq 1$), there exists another $(k\pm 1)$-saddle $y$ such that $x$ and $y$ are connected through a gradient flow. Specifically, if there is no $(k-1)$-saddle connecting $x$, the Morse's homology theorem implies that $x \in \ker \partial_k = \mathrm{im} \partial_{k+1}$ such that $x$ is connected to some $(k+1)$-saddle.

\section{Numerical Experiments}
In this section, we conduct two illustrative examples by comparing different methods to demonstrate the advantages of  the iHiSD in locating saddle points and constructing complex solution landscapes. 

\subsection{A two-dimensional example}
\label{sec:butterfly}
% In \cite{gad}, GAD is tested on the 2D double-well energy function $E(x,y)=\frac14(x^2-1)^2+\frac12\mu y^2$.
% Further analysis \cite{Ortner} showed that the attractive region of GAD is critically influenced by those points with diagonal Hessian matrices.
It has been pointed out in \cite{Ortner} that the  attractive region of the GAD method is significantly influenced by points characterized by diagonal Hessian matrices. 
As a generalization of GAD to higher indices, HiSD also faces this limitation, and simply implementing upward and downward searches with HiSD may restrict the thorough exploration of the solution landscape. 
In contrast, the iHiSD provides a possibility of circumventing this constraint.  To give a clear perspective, we conduct tests using the following two-dimensional energy function
\begin{equation*}
\label{func_butterfly}
    % E(x,y)=x^4-1.5x^2y^2+y^4-1.8y^3+y^2+x^2 y-2x^2.
    E(x,y)=x^4-2x^2+y^4+y^2-1.5x^2y^2+x^2y-cy^3,
\end{equation*}
where $c$ is a parameter. While locating the index-$1$ saddle point, the HiSD degenerates to the GAD such that the iHiSD method is a crossover dynamics from the gradient flow to the GAD.

\begin{figure}[H]
    \centering
    \includegraphics[width=0.5\linewidth]{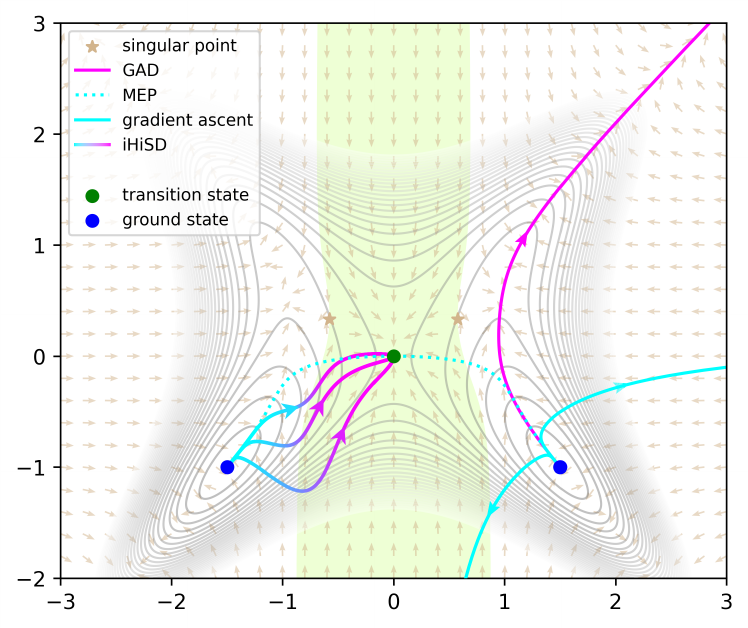}
    \caption{Comparison of different dynamics for searching saddle points starting from a local minimum for $c=1$.
    The GAD method could not cross the boundary of the region of attraction, while the gradient ascent method diverges rapidly due to numerical instability. 
  In contrast, the iHiSD method could traverse this boundary to locate the saddle point under different initial values of $\alpha$.}
    \label{fig:butterfly_1}
\end{figure}

In Figure \ref{fig:butterfly_1}, we select $c=1$ and compare different methods, where there are two symmetric ground states (marked in blue) and a transition state (marked in green) between them. 
The brown arrows represent the GAD vector field generated by
\begin{equation*}
    -\left(I-2v(x)v(x)^T\right)\nabla E(x),
\end{equation*}
where $v(x)$ is the normalized eigenvector corresponding to the smallest eigenvalue of $G(x)$.
The points where the Hessian matrices are diagonal, marked by stars, are singular with respect to the GAD vector field and define the boundary of the region of attraction (highlighted in green) associated with the saddle point, as discussed in \cite{Ortner}. 
Thus, one could observe that the GAD trajectory could not cross the boundary of the region of attraction to reach the saddle point.
It is worth mentioning that this boundary arises from the mechanism of the GAD algorithm, rather than any inherent property of the underlying energy function. 
Nevertheless, this boundary severely limits the effectiveness of GAD, especially when the initial condition falls outside this boundary, preventing the system from fully exploring the space. Furthermore, the  gradient ascent method starting from a local minimum quickly diverges due to numerical instability.

The iHiSD intends to address these limitations through a dynamic parameter $\alpha(t)$, which transitions smoothly from $\alpha=0$ (gradient flow) to $\alpha=1$ (GAD) over time, as indicated by the crossover change of the color in the lines of iHiSD in Figure \ref{fig:butterfly_1}.
Specifically, we select $\alpha(t)$ in \eqref{iHiSD_alpha} as $\dot{\alpha}=\beta\cdot2\alpha(1-\alpha)$ where $\beta=1$.
The trajectories start from perturbations of the ground state $x^*$ such that $|x(0)-x^*|=10^{-2}$ under different initial values of $\alpha(t)$, i.e. $\alpha(0)=10^{-11},10^{-9},10^{-7}$.
 Figure \ref{fig:butterfly_1} indicates that the integration of both the gradient flow and the GAD allows iHiSD to bypass the boundaries of GAD under different initial values of $\alpha(t)$, enabling a more thorough and stable exploration of the space.

With the above demonstrations, we could now construct the solution landscape--whether from a local minimum or a local maximum--using iHiSD, as illustrated in Figure \ref{fig:butterfly_SL}, which demonstrates the ability of iHiSD in traversing the GAD boundary of the region of attraction to construct the solution landscape.

\begin{figure}[H]
    \begin{minipage}{0.48\textwidth}
        \includegraphics[width=\linewidth]{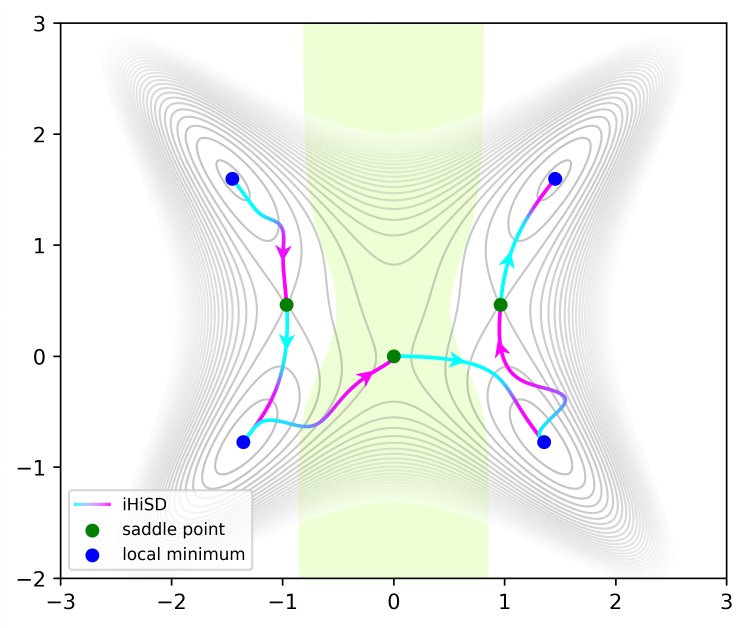}
    \end{minipage}
    \begin{minipage}{0.48\textwidth}
        \includegraphics[width=\linewidth]{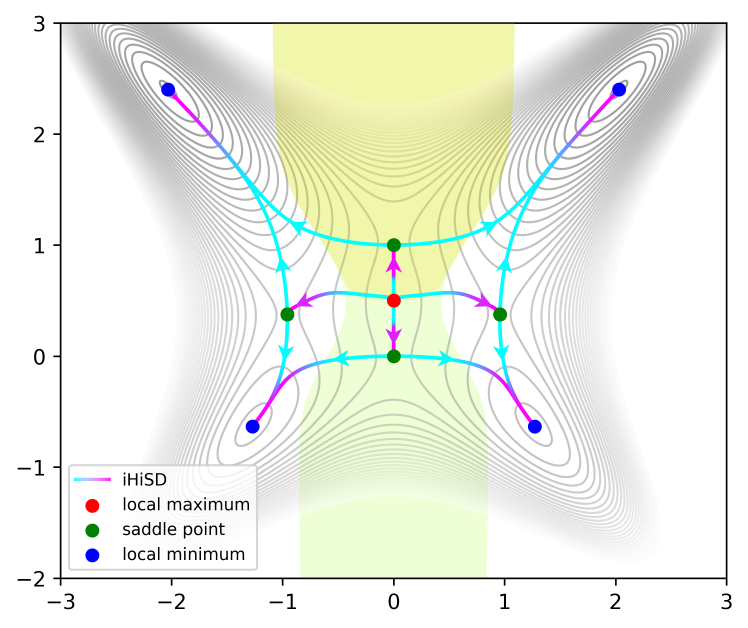}
    \end{minipage}
    \caption{Construction of solution landscape with iHiSD starting from a minimum ($c=1.5$) (left) and  a maximum ($c=2$) (right).}
    \label{fig:butterfly_SL}
\end{figure}

\subsection{Particle Clustering with Morse Potential}
The Morse potential approximates diatomic interactions and is given by the dimensionless formula
$$ V(r) = e^{-2a(r-1)} - 2e^{-a(r-1)}, $$
where $ a > 0 $ denotes the rigidity parameter. 
Notably, for $ a = 6 $, using the approximation $ r \approx 1 + \ln(r) $, the Morse potential simplifies to the well-known Lennard-Jones potential (see Chapter 11 in \cite{Atkins2006}):
$$ V_{LJ}(r) = r^{-12} - 2r^{-6}. $$
This relationship underscores the connection between these potentials and their significance in modeling diatomic interactions.
We consider a planar cluster of $N$ identical particles, with the internal energy expressed as
\begin{equation*}
    E(r_1,\cdots,r_N)=\sum_{1\le i<j\le N}V\left(\|r_i-r_j\|\right),\quad r_i\in\mathbb{R}^2.
\end{equation*}
Although $E$ is not inherently a Morse function on $(\mathbb{R}^2)^N$ due to its translation and rotation invariance, we could consider the quotient space 
%$(\mathbb{R}^2)^N/\sim$, 
where $E$ remains a Morse function. 
To avoid any ambiguity, we will use the term \textit{pattern} to refer to the equivalence class of $(r_1, \ldots, r_N)$ in $(\mathbb{R}^2)^N$.

\begin{figure}[H]
    \centering
    \includegraphics[width=0.6\linewidth]{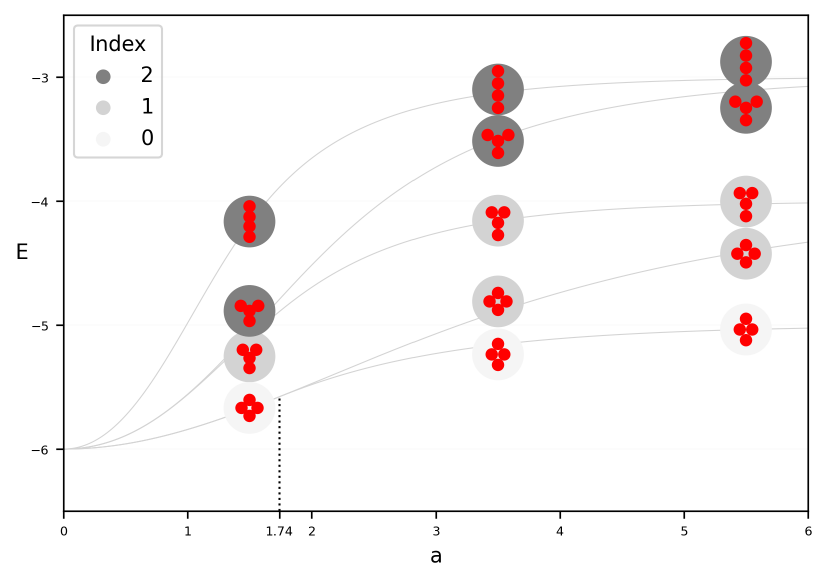}
    \caption{Stationary patterns of Morse potential for different parameter $a$.
    The supercritical pitchfork bifurcation occurs at $a\approx 1.74$.}
    \label{fig:bifurcation}
\end{figure}

For $N=4$, it is straightforward to exhaustively identify all stationary patterns for various values of the parameter $a$, as illustrated in Figure \ref{fig:bifurcation}.
When $a<1.74$, there are 4 stationary patterns: the line and star patterns are $2$-saddles, the fork pattern is a $1$-saddle, and the square pattern is the minimum.
As $a$ increases, the particles tend to aggregate less compactly.
When $a$ exceeds about $1.74$, the attractive potential causes the square pattern to bifurcate into an index-1 pattern, with the diamond pattern emerging as the new minimum, resulting in five stationary patterns.
With greater rigidity, the patterns become more regular, with each particle maintaining an approximate distance of 1 from some others, corresponding to the stable distance of diatomic interactions.

Using the iHiSD method, we construct solution landscapes for the particle clustering with $N=4$ under different parameters in Figures \ref{fig:sl1.5}--\ref{fig:sl6}. 
Although the particles are identical, they are marked with different colors to enhance the visualization of transition pathways. 
In particular, the connections indicated by dashed lines are not captured by the HiSD method but successfully identified by the iHiSD method, which again substantiates the effectiveness of the iHiSD method.

\begin{figure}[H]
    \centering
    \includegraphics[width=\linewidth]{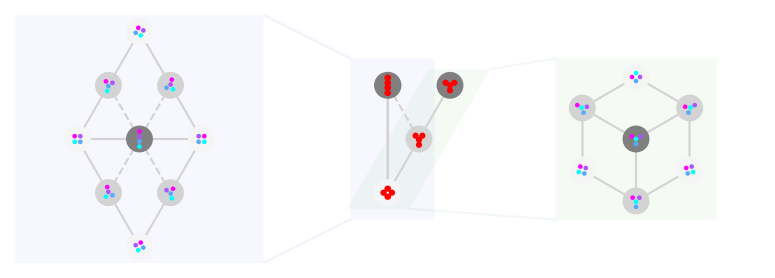}
    \caption{Solution landscape for the particle clustering with $a=1.5$.}
    \label{fig:sl1.5}
\end{figure}

\begin{figure}[H]
    \centering
    \includegraphics[width=\linewidth]{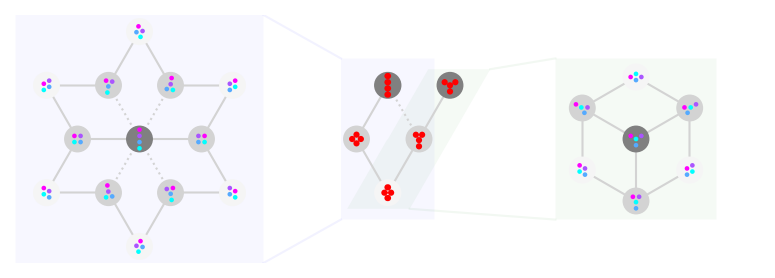}
    \caption{Solution landscape  for the particle clustering with  $a=6$.}
    \label{fig:sl6}
\end{figure}

\section{Concluding remarks}

In this paper, we propose and analyze the iHiSD, which enhances the traditional HiSD by integrating the gradient flow.
Our analysis shows that the HiSD operates as a quasi-Newton method and is limited by the restricted exploration region, which may result in an incomplete solution landscape. 
In contrast, the iHiSD with appropriate weight functions $\beta_1(t)$ and $\beta_2(t)$ could exhibit a stable and nonlocal convergence to the target saddle point, which distinguishes the iHiSD from traditional HiSD. In particular,
the iHiSD enables successive searches between saddle points with the assistance of the Morse theory, which helps to construct a more complete solution landscape. 
Different methods are compared in numerical experiments to substantiate the effectiveness of the proposed iHiSD method.

A potential improvement of the current method is to develop an adaptive strategy for iHiSD, which will adjust the dynamics of weight functions based on the current state and the intrinsic properties of the energy functions. 
This aims to improve the capability of the method in constructing complex energy landscapes and thus exploring complex systems.
\section*{Acknowledgments}
We thank Shuonan Wu for helpful discussions.

\bibliographystyle{siam}
\bibliography{ref}

\end{document}